%% file: Main.tex
\documentclass[12pt]{amsart}
\usepackage{amsthm,amstext,amsmath,amscd,amssymb,latexsym,mathrsfs}

\usepackage{color}
\usepackage{amsfonts,array}
\usepackage{todonotes}
\usepackage{verbatim}
\usepackage{float}	
\usepackage{pgf,tikz}
\usepackage[page]{appendix}
\usepackage[all,cmtip]{xy}
\usepackage{stmaryrd}
\usepackage{xparse}
\usepackage{listings}
\usepackage{enumitem}
\setlist{nosep, left=0pt}

\usepackage{tabularx,booktabs}

\usepackage{ragged2e} 
\usepackage{needspace}   
\usepackage{ltablex}   
\keepXColumns         
\usepackage{placeins} 

\newcolumntype{L}{>{\RaggedRight\arraybackslash}p{3.5cm}}
\newcolumntype{Y}{>{\RaggedRight\arraybackslash}X}

\usepackage{caption}
\usepackage{subcaption}
\usepackage[T1]{fontenc}
\usepackage{lmodern}
\usepackage{etoolbox}
\usepackage{xcolor}

\usepackage[shortalphabetic]{amsrefs}
\usepackage[all]{xy}
\usepackage{newtxtext}
\usepackage{newtxmath}
\usepackage{framed}
\usepackage{graphicx}
\usepackage{wasysym}
\usepackage{pdfpages}

\usepackage[thinlines,thiklines]{easybmat}

\usepackage{microtype} 

\usepackage{mathtools}  

\usepackage{bbm}

\makeatletter
\def\thm@space@setup{%
  \thm@preskip=5pt \thm@postskip=5pt
}
\makeatother

\makeatletter
\renewcommand\@biblabel[1]{[#1]\quad}%

\@ifpackageloaded{amsrefs}{%
  \ifcsname BibLabel\endcsname
    \usepackage{etoolbox}
    \patchcmd{\BibLabel}{\hfill}{}{}{}%
    \renewcommand{\BibLabel}{%
      \ifcsname Hy@raisedlink\endcsname
        \Hy@raisedlink{\hyper@anchorstart{cite.\CurrentBib}\hyper@anchorend}%
        [\thebib]\quad
      \else
        [\thebib]\quad
      \fi
    }%
  \fi
}{}%
\makeatother

\def\Bl{\operatorname{Bl}}

\numberwithin{equation}{section}

\newcommand{\equ}{\ensuremath{\,=\,}}

\newcommand{\deq}{\ensuremath{\stackrel{\textrm{def}}{=}}}

\DeclareMathOperator{\pr}{pr}

\DeclareMathOperator{\lint}{int}

\newcommand{\Ra}{`\ensuremath{\Rightarrow}'  }
\newcommand{\La}{`\ensuremath{\Leftarrow}'  }

\newcommand{\BA}{{\mathbb{A}}}

\newcommand{\BC}{{\mathbb{C}}}

\newcommand{\BG}{{\mathbb{G}}}

\newcommand{\BN}{{\mathbb{N}}}

\newcommand{\BP}{{\mathbb{P}}}
\newcommand{\BQ}{{\mathbb{Q}}}
\newcommand{\BR}{{\mathbb{R}}}

\newcommand{\BZ}{{\mathbb{Z}}}

\newcommand{\Ff}{{\mathfrak{f}}}

\newcommand{\CA}{{\mathcal A}}
\newcommand{\CB}{{\mathcal B}}
\newcommand{\CC}{{\mathcal C}}
\renewcommand{\CD}{{\mathcal D}}

\newcommand{\CF}{{\mathcal F}}

\newcommand{\CO}{{\mathcal O}}

\newcommand{\CS}{{\mathcal S}}

\DeclareMathOperator{\Spec}{Spec}
\DeclareMathOperator{\Proj}{Proj}

\newcommand{\bProj}{\mathop{\text{\textbf{Proj}}}\nolimits}

\DeclareMathOperator{\im}{im}

\DeclareMathOperator{\coker}{coker}

\DeclareMathOperator{\rank}{rank}

\DeclareMathOperator{\Cl}{Cl}

\DeclareMathOperator{\Cone}{Cone}

\DeclareMathOperator{\Cox}{Cox}

\DeclareMathOperator{\Rel}{Rel}
\DeclareMathOperator{\Hom}{Hom}

\newcommand{\Gen}{{\rm Gen}}
\newcommand{\Rat}{{\rm Rat}}

\makeatletter
\def\thm@space@setup{%
  \thm@preskip=5pt \thm@postskip=5pt
}
\makeatother

\newtheorem*{theorem*}{Theorem}

\newtheorem{theorem}{Theorem}[section]
\newtheorem{lemma}[theorem]{Lemma}

\newtheorem{defprop}[theorem]{Definition and Proposition}

\newtheorem{proposition}[theorem]{Proposition}
\newtheorem{corollary}[theorem]{Corollary} 

\theoremstyle{definition}
\newtheorem{definition}[theorem]{Definition}
\newtheorem{example}[theorem]{Example}

\newtheorem{remark}[theorem]{Remark}

\newtheorem{introthm}{Theorem}

\def\phi{\varphi}
\def\epsilon{\varepsilon}
\def\setminus{\smallsetminus}
\let\oldbullet\bullet
\def\bullet{{\mathchoice{\oldbullet}%
                        {\oldbullet}%
                        {\scriptscriptstyle\oldbullet}%
                        {\oldbullet}}}

\let\emptyset\varnothing

\definecolor{GoetheBlue}{RGB}{0,97,143}
\usepackage[bookmarks=true, colorlinks=true, linkcolor=blue!50!black,
citecolor=GoetheBlue, urlcolor=orange!50!black, pdfencoding=unicode]{hyperref}
\usepackage[left=3.5cm,right=3.5cm,top=3cm,bottom=3cm,footskip=1cm
]{geometry}

\title{Relevant Maps and the algebraic skeleton of simplicial toric prevarieties} 

\author{Felix G\"obler}

\begin{document}





\input{abstract}

\maketitle

\tableofcontents


\section*{Introduction}
Toric varieties are known for their rich structure and are easily accessible via their combinatorial description in terms of fans.
In 1992, David Cox constructed coordinate rings for toric varieties, showing that the coordinate ring of a toric variety $X_\Sigma$ is a polynomial ring graded by a finitely generated abelian group (the divisor class group of $X_\Sigma$, see \cite{Cox}). Almost ten years later, Brenner--Schröer presented a Proj construction for rings graded by finitely generated abelian groups, and proved that for finite polynomial rings, the resulting scheme is a toric prevariety (cf.\ \cite{BS}).

The aim of this paper is twofold:
On one side, we want to understand the translation from polynomial rings (with irrelevant ideal) to systems of fans in both directions.
On the other side, we want to use this dictionary to understand toric varieties algebraically, i.e.\ in terms of their coordinate rings.
To do so, we have to consider several kinds of 'permissible` mappings between rings graded by finitely abelian groups.

Let $D$ be a finitely generated abelian group and $S$ a commutative unital $D$-graded ring. A homogeneous element $f \in S$ is called \emph{relevant} if the group of degrees of units $D^f$ of $S_f$ has finite index in $D$. For basic properties of relevant elements, as well as structural theorems on the \emph{multigraded Proj} construction
\begin{align*}
    \Proj^D(S) \equ \bigcup_{f \text{ relevant}} \Spec(S_{(f)}), 
\end{align*} 
we refer to our previous papers \cite{paper1}, \cite{paper2}, and \cite{paper3}.

Aiming to describe morphisms between Proj schemes, we analyze when relevant elements in multigraded rings are mapped to relevant elements. Since even lower requirements are sufficient to define a morphism or a rational map of multigraded spectra, we will also take a closer look at those special cases.

As an application of these notions of maps, we want to describe the category of simplicial toric prevarieties algebraically.
Given a noetherian multigraded polynomial ring $S$, \cite{BS} showed that $\Proj^D(S)$ is a simplicial toric prevariety with torus $G = \Spec(S_0[D_F])$ ($D_F$ being the free part of $D$, see also Theorem~\ref{thm:ring_to_prev}).
Generalizing the ideas of \cite{Cox} to toric prevarieties, we show that a simplicial toric prevariety $X$ yields a multigraded polynomial ring $S$, that is related to, but not necessarily identical with, the \emph{Cox ring} of $X$.

If we extend the data of the multigraded polynomial ring $S$ by an `irrelevant subset', that is, a subset $B \unlhd S$ of the full irrelevant ideal $S_+$, the abovementioned functors give rise to an anti-equivalence of categories. In particular, conical rings serve as a generalization of the so-called  \emph{bunched rings}, a notion introduced by \cite{CRB}.

\begin{introthm}[\autoref{thm:functor_conical_rings_toric_prev}]
    Let $S$ be a noetherian factorially and effectively $D$-graded polynomial ring over an algebraically closed field.
  Then the category $\mathrm{ToricPrev_s}$ of simplicial toric prevarieties (with morphisms of toric prevarieties as morphisms) is anti-equivalent to the category $\mathrm{RatConRing}$ of conical rings over noetherian polynomial rings together with rational maps of conical rings.
    In particular, multigraded noetherian polynomial rings serve as coordinate rings for toric prevarieties.
\end{introthm}

While aspects of toric maps and Cox rings have, of course, been studied before, we are not aware of a detailed account in terms of multigraded rings, nor of one that incorporates Brenner--Schröer’s notion of relevance.
We generally assume that toric (pre)varieties are normal, as in \cite{CLS}.


{\normalfont\bfseries Organization of the article.}
In Section~\ref{sec:mult_proj}, we quickly recall the definition of relevant elements, before diving into several notions of relevant maps between multigraded rings. At first, we will look at those maps without considering the irrelevant ideal as an extra datum.
Section~\ref{sec:toric_prev} is mostly devoted to the construction of an affine system of fans from a multigraded ring and vice versa. The ideas are already present in \cite{Cox} and \cite{BS}, but the correspondence has not yet been settled.
Finally, we will show that the Chow quotient of a toric prevariety remains a toric prevariety.\\


{\normalfont\bfseries Acknowledgments.}
I would like to thank my (former) supervisor, Alex Küronya, for his tremendous support and insightful feedback.
I am also very grateful to Johannes Horn, Andrés Jaramillo Puentes, Kevin Kühn, Jakob Stix, Martin Ulirsch, and Stefano Urbinati for many helpful discussions and valuable comments.

Partially funded by the Deutsche Forschungsgemeinschaft (DFG, German Research Foundation) TRR 326 \textit{Geometry and Arithmetic of Uniformized Structures}, project number 444845124, and by the LOEWE grant \emph{Uniformized Structures in Algebra and Geometry}.
















\section{Functoriality of multigraded Proj}\label{sec:mult_proj}
\input{Functoriality}


\section{Simplicial Toric (Pre)varieties}\label{sec:toric_prev}
\input{Pre_Toric_Proj}

\input{Bibliography}
\end{document}

%% file: abstract.tex
\begin{abstract}
Morphisms between schemes arising from multigraded rings are essential for understanding geometric relationships in algebraic geometry, yet a systematic theory for such maps has been lacking. 
In this paper, we develop a comprehensive framework for rational maps between multigraded Proj schemes by introducing several notions of maps between their underlying multigraded rings. A key challenge is that to induce actual morphisms (rather than just rational maps), the ring homomorphism $\varphi\colon R \to S$ must hit every relevant element in $S$.
To address this, we introduce the use of relevant subsets $B \subseteq S_+$ (where $S_+$ is the ideal generated by all relevant elements), $B \unlhd S$, which allow us to control this condition more flexibly. 

As an application, we show that multigraded noetherian polynomial rings naturally encode combinatorial data, giving rise to systems of fans and thus to toric prevarieties. By leveraging our notion of rational maps with those relevant subsets, we prove that the category of triples $(D, S, B)$ - where $D$ is a finitely generated abelian group, $S$ is a $D$-graded noetherian polynomial ring, and $B \unlhd S$ is a subset of $S_+$ - together with rational maps of conical rings, is anti-equivalent to the category of simplicial toric prevarieties.

\end{abstract}

%% file: Functoriality.tex
Let $D$ be a finitely generated abelian group.
First, we need to recall several definitions from our previous papers.
A commutative unital ring $S$ with a grading by a finitely generated abelian group is called a \emph{multigraded ring}.
We overall assume that multigraded rings are effectively graded, so that the grading group is generated by its support (cf. \cite{paper1}, Example 1.12).

\begin{definition}
    Let $R$ be $D_R$-graded and $S$ be $D_S$-graded rings, where $D_R$ and $D_S$ are finitely generated abelian groups. A \emph{morphism of multigraded rings} is a pair $(\varphi, \alpha_\varphi)$, where $\varphi$ is a ring homomorphism $\varphi\colon R\to S$, and $\alpha_\varphi$ is a group homomorphism $\alpha_\varphi\colon D_R \to D_S$ satisfying $\varphi(R_d) \ \subseteq \ S_{\alpha_\varphi(d)}$ for all $d\in D_R$. If $D_R=D_S$, we will also call $\varphi$ a \emph{morphism of $D_R$-graded rings}. If $\varphi$ is an isomorphism, but $D_R \neq D_S$, we will call $\varphi$ (or rather the induced map $\alpha_\varphi\colon D_R \to D_S$) a \emph{regrading} of $S$. 
\end{definition}

 For homogeneous $f \in S_d$, we its weight cone 
 \begin{align*}
    \CC_D(f) &\deq \overline{\Cone}( \{\deg_D(g) \mid g \text{ is a homogeneous divisor of $f^k$ for some $k\ge 0$}\}) 
\end{align*}
in $D_\BR = D \otimes_\BZ \BR$. Recall that $f \in S_d$ is called \emph{relevant}, if, for example, $\lint(\CC_D(f)) \subseteq D_\BR$ is non-empty, or, if the group of units $D^f$ of $S_f$ has finite index in $D$ (see \cite{paper1}, Lemma 1.15). The ideal $S_+$ generated by all relevant elements in $S$ is called the \emph{irrelevant ideal}, the set of relevant elements is denoted by $\Rel^D(S)$.
If $S$ is factorially graded and noetherian, we can choose a minimal generating system of $S$ by $D$-prime elements and hence construct minimal generators for $S_+$ via monomials in those $D$-prime elements. We denote a corresponding minimal generating set for $S_+$ by $\Gen^D(S)$. For more details, we refer to \cite{paper1}, section 1.2.

\subsection{Relevant Maps}

In the classical Proj setting, rings are $\BN$-graded, and thus, it suffices to study morphisms of $\BN$-graded rings.

\begin{example}\label{ex:maps_special}
    Consider the multigraded rings $R = \BC[x, y, z]$, where $D_R = \BZ^2$, $\deg(x) = (1,0)$, $\deg(y) = (0, 1)$ and $\deg(z) = (1, 1)$ and $S = \BC[x, y, z, w]$, where $D_S = \BZ^3$, $\deg(x) = (1, 0, 0)$, $\deg(y) = (0, 1, 1)$, $\deg(z) = (0, 0, 1)$ and $\deg(w) = (1, 1, 0)$. Then $R_+ = \langle xy, xz, yz \rangle$, and $S_+  = \langle xyz, xyw, yzw, xzw \rangle$.
Now $\Proj^{\BZ^2}(R)$ corresponds to a $\BP^1$ with doubled origin (cf. \cite{paper1}, Example 1.28), whereas $\Proj^{\BZ^3}(S)$ corresponds to a $\BP^1$ with doubled origin and infinity (cf. \cite{ANH}, Example 2.4).

Clearly, there are straightforward inclusions and projections between those two schemes. Hence, we do really need to consider morphisms of multigraded rings, where the grading groups differ.
But this is not the only complication we have to take into account.
For example, the morphism $\varphi\colon R \to S$, $x \mapsto x, y \mapsto y, z \mapsto wz$ is an injective morphism of multigraded rings, which induces $\alpha_\varphi\colon \BZ^2 \to \BZ^3, e_1 \mapsto e_1$ and $e_2 \mapsto e_2 + e_3$. But  $\varphi(R_+) = \langle xy, xwz, ywz \rangle \not\subseteq \Rel^{D_S}(S)$, since $xy$ is not relevant in $S$. But we can view $xy$ as an element of $\varphi(R)$ graded by $\BZ^2$. Thus, in this case, $xy$ is relevant!
\end{example}

An intuitive solution to this problem would be to restrict to the elements that remain relevant.
In the special case where all relevant elements of $S$ lie in the image of $\varphi$, this will induce a morphism of schemes $\Proj^D(S) \to \Proj^D(R)$.
However, note that we only require the local morphisms of multigraded rings $R_{(g)} \to S_{(f)}$ to exist, as they induce a morphism of schemes, as long as $f$ is relevant in $S$ or $f$ divides a relevant element $h$ satisfying $S_{(f)} = S_{(h)}$.
Therefore, we need several notions of morphisms, such that all those situations are reflected.

\begin{definition}[Relevant Morphism]\label{def:rel_loc}
    Let $\varphi\colon R\to S$ be a morphism of multigraded rings. 
    \begin{enumerate}[label=(\arabic*)]
        \item We define the \emph{relevant locus} of $\varphi$ to be
        \begin{align*}
            \Rel(\varphi) \deq  \Rel^{D_S}(S) \cap \varphi(\Rel^{D_R}(R))
        \end{align*}
        and likewise the \emph{generic locus}
        \begin{align*}
            \Gen(\varphi) \deq \Gen^{D_S}(S) \cap \varphi(\Gen^{D_R}(R)).
        \end{align*}

       \item We say that $\varphi$ is a \emph{relevant} morphism of multigraded rings, if 
        \begin{align*}
            \Rel^{D_S}(S) \ \subseteq \ \varphi(\Rel^{D_R}(R)).
        \end{align*}
       \item We define the \emph{rational locus} 
       \begin{align*}
                  \Rat(\varphi) \deq \Rel^{D_R}(R) \cap \varphi^{-1}(\Rel^{D_S}(\varphi(R)))
       \end{align*}
       of $\varphi$ to be the set of all relevant elements that remain relevant under $\varphi$ and denote the \emph{generic rational locus} $\Rat(\varphi) \cap \Gen^{D_R}(R)$ by $\Gen(\Rat)$.

        
        \item We say that $\varphi$ is \emph{rationally relevant}, if
        \begin{align*}
             \Rel^{D_R}(R) \ \subseteq \ \varphi^{-1}(\Rel^{D_S}(\varphi(R)) .
        \end{align*}

        \item We define a \emph{rational map of multigraded rings} $\varphi_r\colon R \dashrightarrow S$ to consist of
        \begin{enumerate}[label=(\roman*)]
            \item A map $R_+ \to h(S)$, such that for all relevant $g \in R_+$, there exists a relevant $f \in S$ such that $\varphi_r(g) \mid f$ satisfying $S_{(f)} = S_{(\varphi_r(g))}$.
            \item A collection of compatible morphisms of multigraded rings $S_{(g)} \to S_{(f)}$ for all $g \in R_+$. 
        \end{enumerate}
    \end{enumerate}
\end{definition}

The term \emph{locus} in Definition~\ref{def:rel_loc} (1) is justified, as $\Rel(\varphi)$ as well as $\Gen(\varphi)$ give rise to open subsets. Concretely,
    \begin{align*}
        U_\text{Rel} &\deq \bigcup_{f \in \Rel(\varphi)} \Spec(S_{(f)}) \subseteq \Proj^D(S) \text{ and } \\
        U_\text{Gen} &\deq \bigcup_{f \in \Gen(\varphi)} \Spec(S_{(f)}) \subseteq \Proj^D(S).
    \end{align*}
The rational locus defines an open subset of $\Proj^{D_R}(R)$, which is
    \begin{align*}
        U_\Rat \deq \bigcup_{g \in \Rat(\varphi)} \Spec(S_{(g)} \subseteq \Proj^{D_R}(R).
    \end{align*}
    Likewise, the image of the rational locus under $\varphi$ defines an open subset of $\Proj^{D_S}(S)$.

\begin{remark}
    A rationally relevant morphism $\varphi \colon R \to S$ gives rise to a rational map $r\colon \Proj^{D_S}(S) \dashrightarrow \Proj^{D_R}(R)$, as it defines a morphism on the open subset $\cup_{f \in \Rel^D(S) \cap \varphi(R_+)} \Spec(S_{(f)})$ of $\Proj^{D_S}(S)$, that is given by all induced local maps $S_{(\varphi(g))} \to R_{(g)}$. In particular, every rationally relevant map of multigraded rings gives rise to a rational map of multigraded rings.
\end{remark}

\begin{example}
    In the situation of Example~\ref{ex:maps_special}, it holds 
    that $\varphi \colon R \to S$ is rational but neither relevant nor rationally relevant. But, if we view $\varphi$ as $\varphi \colon R \to \varphi(R)$, it is relevant.
\end{example}

It is no surprise that each relevant map is also rationally relevant (and therefore also gives rise to a rational map of multigraded rings):

\begin{lemma}\label{lem:rel_loc_im}
    Let $\varphi\colon R\to S$ be a morphism of multigraded rings and $f \in S$ relevant. If there is a homogeneous $g \in R$ such that $\varphi(g) = f$, then $g$ is relevant in $R$.  In particular, if all relevant elements of $S$ are hit, then all relevant elements of $R$ are mapped to relevant elements of $S$.
\end{lemma}

\begin{remark}
    Note that the compatibility condition $\varphi(R_d) \ \subseteq \ S_{\alpha_\varphi(d)}$ implies that for all relevant $g \in R$ there is a relevant $f \in S$ such that $\varphi(g) \mid f$.
\end{remark}

\begin{proof}
    We have a induced map $\varphi_g \colon R_g \to S_f$, where $\alpha_{\varphi_g}$ is surjective, as $S_f$ is periodic. Thus all elements $\deg(h)$ for $h \in (S_f)^\times$ have a preimage in $D_R$. The corresponding elements $g' \in (S_g)^\times$ are mapped to $(S_f)^\times$, i.e. $D^f \subseteq \alpha_{\varphi_g}(D^g) \subseteq D_S$. Thus, $D^g \subseteq D_R$ has to have finite index and $g$ is relevant.
\end{proof}

The actual notion of relevance for maps is a bit unsatisfying, as it coincides with the map being surjective. But we are only interested in the fact that the irrelevant ideal lies in the image. However, we might achieve this by just changing the irrelevant ideal.
For capturing this subtlety, we will define the category of \emph{conical rings} in Section~\ref{sec:conical rings}.

\begin{corollary}
Let $\varphi \colon R\to S$ be a morphism of multigraded rings.
\begin{enumerate}[label=(\roman*)]
    \item $\varphi$ is relevant when viewed as morphism $R \to \varphi(R)$.
    \item $\varphi$ is relevant if and only if $\varphi$ is surjective.
\end{enumerate}
\end{corollary}

\begin{proof}
    We only have to show that a relevant map must be surjective. Assume there is a homogeneous $g \in S$ that has no preimage in $R$. Let $f \in S$ be relevant such that $g \mid f$. Now (as $\varphi$ is assumed to be relevant) there exists $f' \in R$ relevant such that $\varphi(f') = f$. Since $\varphi$ is a morphism of multigraded rings, one of the divisors of $f'$, say $g'$, must be mapped to $g$. Thus $\varphi$ is surjective.
\end{proof}

In the remaining part of the section, we will discuss when we can ensure that relevant elements are mapped to relevant elements. In this case, we can always restrict the irrelevant ideal in the target to the image of the relevant elements that remain relevant (giving rise to a rationally relevant map).

\begin{proposition}\label{thm:morph_mult_relevant}
    Let $\varphi\colon R \to S$ be a morphism of multigraded rings.
    Then $\varphi$ maps relevant elements to relevant elements if and only if $\alpha_\varphi$ is surjective.  
\end{proposition}

\begin{proof}
\La Let $g \in R$ be relevant. We have to show that $f = \varphi(g) \in S$ is also relevant. The degrees of the homogeneous divisors of $g$ generate a maximal-dimensional cone. Now by surjectivity
\begin{align*}
    \dim_\BR( \underbrace{\alpha_{\varphi, \BR} (\CC_{D_R}(g))}_{= \CC_{\alpha_\varphi}(f)}) \equ \rank(D_S),
\end{align*} 
hence $f$ is relevant. \\
\Ra Let $g \in R$ and $f = \varphi(g) \in S$ be relevant. We are done if we show that the degrees of the homogeneous divisors of $f$ are in the image of $\alpha_\varphi$. By the compatibility condition on the gradings and since $\varphi$ is a homomorphism of rings, the homogeneous divisors of $f$ are the images of the homogeneous divisors of $g$. In particular, the generators of $\CC_{D_S}(f)$ are exactly the generators of $\CC_{D_R}(g)$ that do not lie in $\ker(\alpha_\varphi)$. Hence $\alpha_{\varphi, \BR}$ maps maximal-dimensional cones to maximal-dimensional cones and thus $\alpha_\varphi$ is surjective
\end{proof}

For polynomial rings with equal grading group, this result simplifies to:

\begin{corollary}\label{thm:special_morph_relevant}
Let $R, S$ be multigraded polynomial rings with finitely many variables, both graded by the group $D$.
    Let $\varphi \colon R \to S$ be a morphism of $D$-graded rings. Then $\alpha_\varphi$ is surjective. In particular, every morphism of $D$-graded rings maps relevant elements to relevant elements.
\end{corollary}

\begin{proof}
    Let $f \in R$ be relevant. Then $\CC_{D}(f)$ is polyhedral by \cite{paper1}, Lemma 1.20, say generated by $c_1, \ldots, c_r \in \BZ^r$. Since $\varphi(R_{c_i}) \subseteq S_{\alpha_\varphi(c_i)}$, $\CC_{D}(\varphi(f))$ has to be generated by $c_i' = \alpha_\varphi(c_i)$, i.e.\ $\CC_{D}(\varphi(f))$ has to be maximal-dimensional. In particular, $\varphi(f)$ is relevant.
\end{proof}

\begin{example}
    In the situation of Example~\ref{ex:maps_special}, clearly $\varphi \colon R\to S$ does not map relevant elements to relevant elements, and $\alpha_\varphi$ is not surjective. But viewed as map $\varphi'\colon R \to \varphi(R)$, where $S' = \varphi(R) = \BC[x, y, zw]$, we can see that $\alpha_{\varphi'}$ is surjective, as $S'$ can be seen as $\BZ^2$-graded ring, being isomorphic to $R$ as graded ring. In particular, $\Proj^{D_S}(S)$ and $\Proj^{D_R}(R)$ are birational.
\end{example}


\subsection{Conical Rings and Functoriality}\label{sec:conical rings}

It is apparent that $\varphi\colon R \to S$ being relevant depends on the irrelevant ideals $R_+$ and $S_+$. 
If we take certain subsets of the full irrelevant ideal into account, i.e.\ subsets $B_R \subseteq R_+$ that are also ideals $B_R \unlhd R$ such that $\varphi(B_R) \subseteq S_+$, we get a more flexible notion.
In particular, we can force $\varphi$ to be relevant by choosing a subset of the irrelevant ideal, that is, choosing a subset of relevant elements.

We want to make this choice to be a part of the data in order to be able to distinguish objects on the ring side. 
The resulting space is always an open subset in the respective Proj by definition, as it is given by gluing basic affine open subsets of Proj.
Since the ring associated to a system of fans will not be unique, we establish so-called \emph{conical rings} and their rational maps to describe such situations.

The terminology is motivated by the observation that each $\Spec(S_{(f)})$ corresponds to a convex cone, when $S$ is a noetherian polynomial ring. In particular, a subset $B \subseteq S_+$ corresponds to a system of cones, giving rise to a system of fans (see Section~\ref{sec:sys_fans}), and in special cases, to a fan. 

The material in this section has not, as far as we are aware, been treated systematically before. We incorporate the choice of relevant subsets into the ring data and develop morphisms, rational maps, and birationality in this context.

\begin{defprop}[Category of conical rings]\makebox{}{}\label{def:cat_con}
Let $S$ be a multigraded ring with grading group $D_S$.
Recall that $h(S)$ denotes the set of homogeneous elements of the multigraded ring $S$ and $\Rel^D(S)$ denotes the set of relevant elements of $S$. 
\begin{enumerate}[label=(\arabic*)]
    \item We call a pair $(S, B)$ a \emph{conical ring}, if $B \unlhd S$ is a homogeneous ideal satisfying $B \subseteq  S_+$.

    \item We define the set of $B$-relevant elements as
    \begin{align*}
        \Rel^{D_S}_B(S) \deq \Rel^{D_S}(S) \cap B.
    \end{align*}

    \item Let $\varphi\colon R\to S$ be a morphism of multigraded rings and $(R, B_R)$, $(S, B_S)$ conical rings. Then we define the \emph{$B_S$-relevant locus of $\varphi$} to be
    \begin{align*}
        \Rel_{B_S}(\varphi)\equ \{f \in \Rel^{D_S}_{B_S}(S) \mid \exists g \in \Rel^{D_R}_{B_R}(R): \varphi(g) = f\}.
    \end{align*}

    \item We define a \emph{morphism} $\Phi \colon (R, B_R) \to (S, B_S)$ of \emph{conical rings} to be a morphism of multigraded rings $\varphi\colon R \to S$ that is $B_S$-relevant, i.e.
    \begin{align*}
        B_S \equ  \Rel_{B_S}(\varphi).
    \end{align*}
    
   \item  An isomorphism of conical rings is a morphism $\Phi \colon (R, B_R) \to (S, B_S)$ of conical rings that admits an inverse $\Psi \colon (S, B_S) \to (R, B_R)$ that is also a morphism of conical rings.

   \item We define a \emph{rationally relevant morphism} of conical rings $\Phi\colon (R, B_R) \to (S, B_S)$ to be a morphism of multigraded rings $\varphi\colon R \to S$ such that $\Rat(\varphi) = B_R$. In other words, $\Phi$ defines a morphism $\Phi'\colon (R, B_R) \to (\varphi(R), \varphi(B_R))$ of conical rings.
   

   \item  We define a \emph{rational map} $\Phi \colon (R, B_R) \dashrightarrow (S, B_S)$ of conical rings to be 
    \begin{enumerate}[label=(\roman*)]
        \item a map $\varphi\colon B_R \to h(S)$, such that for all relevant $g \in B_R$ there exists a relevant $f \in B_S$ satisfying $\varphi(g) \mid f$ and $S_{(f)} = S_{(\varphi(g))}$.

        \item a collection of morphisms of multigraded rings $\phi_g \colon R_{(g)} \to S_{(f)}$ for all $g \in B_R$ relevant that are compatible on overlaps.
    \end{enumerate}   
    \item A rational map $\Phi \colon (R, B_R) \dashrightarrow (S, B_S)$ of conical rings is called \emph{birational}, if there exists $B'_S \subseteq S_+$ such that for all relevant $g \in B_R$ there exists a unique relevant $f \in B'_S$ with $\varphi(g) = f$ and all local maps $\varphi_{g}$ are isomorphisms (i.e. if $\varphi$ gives rise to an isomorphism $(R, B_R) \to (S', B'_S)$ for some subring $S' \subseteq S$).    
\end{enumerate}

Conical rings together with morphisms of conical rings give rise to a category, which we denote by $\mathrm{ConRing}$. 
More importantly, conical rings with rational maps as morphisms also define a category, which we denote by $\mathrm{RatConRing}$.
\end{defprop}

\begin{remark}\label{rem:choice_fg}
Given a rational map $\Phi:(R,B_R)\dashrightarrow (S,B_S)$, by (i) we may (and will) fix for every relevant $g\in B_R$ an element $f_g\in B_S$ with
\[
\varphi(g)\mid f_g\quad\text{and}\quad S_{(f_g)}=S_{(\varphi(g))}.
\]
If $\Psi:(S,B_S)\dashrightarrow (T,B_T)$ is a second rational map, then applying (i) of $\Psi$ to each $f_g\in B_S$ yields elements $h_{f_g}\in B_T$ with
\[
\psi(f_g)\mid h_{f_g}\quad\text{and}\quad T_{(h_{f_g})}=T_{(\psi(f_g))}.
\]
Hence the local composites $R_{(g)}\xrightarrow{\phi_g}S_{(f_g)}\xrightarrow{\psi_{f_g}}T_{(h_{f_g})}$ are well-defined for every $g\in B_R$, and compatibility on overlaps ensures they glue.
\end{remark}

\begin{proof}
    We have to show that the composition of morphisms is well-defined and associative. Thus let $\Phi\colon (R, B_R) \to (S, B_S)$ and $\Psi\colon (S, B_S) \to (T, B_T)$ be morphisms of conical rings, such that $\varphi\colon R \to S$ and $\psi\colon S \to T$ are morphisms of multigraded rings. For $h \in B_T$ we find $f \in B_S$, for which we find $g \in B_R$ such that $\varphi(g) = f$ and $\psi(f) = h$. In particular, $h = (\psi \circ \varphi)(g)$, i.e. $\psi \circ \varphi$ is relevant.
    The associativity follows from the fact that morphisms of conical rings are morphisms of rings. The functoriality of $\alpha_\varphi$ and $\alpha_\psi$ is trivial. 

    Composition of rational maps follows immediately from Remark~\ref{rem:choice_fg}. Associativity and identities are clear from the corresponding properties of the local maps. As local maps compose and compatibility on overlaps is preserved, $\mathrm{RatConRing}$ forms a category.
\end{proof}

\begin{definition}\label{def:rel_proj}
    Let $(S, B)$ be a conical ring, where $B \subseteq S_+$ is an ideal in $S$. Then we define the open subscheme
    \begin{align*}
        \Proj^D_B(S) \deq \bigcup_{f \in \Rel^D_B(S)} \Spec(S_{(f)}) \ \subseteq\ \bigcup_{f \in \Rel^D(S)} \Spec(S_{(f)}) \equ \Proj^D(S).
    \end{align*}
    Furthermore, if a minimal generating system of $B \cap \Rel^D(S)$ exists, we will denote it by $\Gen^D_B(S)$.
\end{definition}

\begin{remark}
\begin{enumerate}[label=(\arabic*)]
    \item In the $\BN$-graded standard case, i.e.\ for $S = \BC[x_0, \ldots, x_n]$, we do not need to add this extra layer of complexity, as a graded morphism $\varphi\colon S \to R$ automatically defines the image of $S_+ = (x_0,\ldots, x_n)$. Since relevant elements are exactly elements of degree $d$ for $d \neq 0$, we see that for defining a graded morphism of rings, it is enough to ensure that $\varphi(R_+) \subseteq S_+$.
    \item For $B_R = R_+$ and $B_S = S_+$, a morphism of conical rings $\Phi\colon (R, B_R) \to (S, B_S)$ is just a surjective morphism of multigraded rings. On the other hand, the corresponding morphism of multigraded rings $\varphi$ need no longer be surjective in order to be relevant.

    \item If there is more than one $f \in B_S$ such that $\varphi(g) \mid f$ in (7), then each choice of $f$ (if not already in the image) defines a rationally relevant morphism, respectively a rational map.

    \item If $\varphi\colon R \to S$ is rationally relevant such that for all relevant $g \in B_R$ there exists a unique relevant $f \in S$ satisfying $\varphi(g) = f$ and $S_{(f)} = \varphi(R)_{(\varphi(g))}$, then the associated map $\Phi \colon (R, B_R) \to (S, \varphi(B_R))$ is a morphism of conical rings. Note that $\varphi$ restricts to an isomorphism of conical rings $(R, B_R) \to (\varphi(R), \varphi(B_R))$ if $\varphi$ is injective.

    \item If $S$ is factorially graded and noetherian, the conditions in (4), (6), and (7) in Definition~\ref{def:cat_con} only have to be checked on the generic or rational locus.
    \item Birationality defines an equivalence relation, where it holds that 
    \begin{align*}
        (S_1, B_1) \ \sim \ (S_2, B_2) \ \iff \ \Proj^{D_1}_{B_1}(S_1) \equ \Proj^{D_2}_{B_2}(S_2).
    \end{align*}
\end{enumerate}    
\end{remark}

The following Lemma is of great importance. It allows us to identify a conical ring $(S, B)$, with $B \subsetneq S_+$ being properly contained, with a subring $S' \subseteq S$ and its irrelevant ideal $S'_+ = B$.

\begin{lemma}\label{lem:subring_localization}
    Let $(S, B)$ be a conical ring such that $B \subsetneq S_+$ is a proper subset.
    Then there exists a subring $S' \subseteq S$ satisfying $S'_+ = B$ and $S_{(f)} = S'_{(f)}$ for all relevant $f \in B$. In particular,
    \begin{align*}
        \Proj^D_B(S) \equ \Proj^D(S').
    \end{align*}
    Therefore, $(S', S_+) \to (S, B)$ is an injective morphism of conical rings.
\end{lemma}

\begin{proof}
We have to construct $S'$ in such a way that it contains all possible numerators of elements of degree zero with $f \in B$ as denominator (where a lot of factors might cancel out). Therefore we define 
\begin{align*}
    S' \deq S_0\left[\left\{g \in S \mid \deg\left(\frac{g}{f'}\right) = 0 \text{ for some } f \in B \text{ such that } f'\mid f\right\}\right], 
\end{align*}
where we assume all $\frac{g}{f'}$ to be either in reduced fractional form or of type $\frac{f}{f}$. 
Then $B \subset S'$ by construction. As $S'$ contains all homogeneous divisors of all $f\in B$, it is also clear that each $S'_{(f)}$ contains the same rational functions of degree zero as $S_{(f)}$.
Hence it holds
    \begin{align*}
        \Proj^D_B(S) &\equ  \bigcup_{f \in \Rel^D(S) \cap B} \Spec(S_{(f)}) 
        &\equ  \bigcup_{f \in S'_+} \Spec(S'_{(f)}) 
        &\equ \Proj^D(S').
    \end{align*}
\end{proof}

\begin{remark}\label{rem:subring_not_unique}
Note that there might be lots of subrings of $S$ satisfying this property. For example take $S = \BC[x, y,z]$ graded by $e_1$, $e_2$ and $e_1+e_2$, where $S_+ = (xy, xz, yz)$. Take $B = (xy, xz)$ (so that $\Proj^D_B(S) = \BP^1$). Then for $S$, $S' = \BC[xy, x, z]$ and for $S'' = \BC[xy,xz]$ we see that the degree zero localizations at $f = xy, xz$ coincide.
\end{remark}

As morphisms of multigraded rings induce a subring structure, the previous Lemma implies:

\begin{corollary}
    Let $\Phi\colon (R, R_+) \to (S, B)$ be an injective morphism of conical rings, where $B = \varphi(R_+)$. Then 
    \begin{align*}
        \Proj^{D_R}(R) \equ \Proj^{D_S}_B(S) .
    \end{align*}
    In particular, every morphism of conical rings $\Psi\colon (A, B_A) \to (S, B_S)$ induces a morphism $\Psi^\ast \colon \Proj^{D_R}(R) \to \Proj^{D_A}_{B_A}(A)$.
\end{corollary}

We also have to consider maps that give a setting like in Example~\ref{ex:maps_special}.

\begin{corollary}\label{cor:subring_loc}
    Let $\varphi\colon R\to S$ be a morphism of multigraded rings, such that there exists a subset $B_S \subseteq S_+$ for which holds $S_{(f)} = \varphi(R)_{(f)}$ for all relevant $f \in B_S$ and $\varphi^{-1}(B_S) = R_+$. Then $\Phi\colon (R, R_+) \to (S, B_S)$ is a morphism of conical rings that induces an isomorphism
    \begin{align*}
        \Proj^D_B(S) \ \cong \ \Proj^D(\varphi(R)).
    \end{align*}
\end{corollary}


Next, we define the $\Proj^D_B$ functor, which assigns a scheme to a conical ring, a morphism of schemes to a morphism of conical rings, and most importantly, an equivalence of rational maps of conical polynomial rings and morphisms of toric prevarieties (which we will deal with in Section~\ref{sec:toric_prev}).

\begin{proposition}\label{prop:Proj_B^D}
There is a functor into the category of schemes
\begin{align*}
    \bProj\colon \textbf{ConRing} \longrightarrow \textbf{Sch},\ (S, B) \mapsto \bigcup_{f\in\Rel^{D_S}_B(S)} \Spec(S_{(f)}),
\end{align*}
where $\Spec(S_{(f)})$ and $\Spec(S_{(g)})$ are glued along the open subscheme $\Spec(S_{(fg)})$ (c. f. \cite{May}, Proposition 1). If $S$ is factorially graded and noetherian, we might take the union over all $f\in\Gen^D_B(S)$. \\

A morphism $\Phi\colon (R, B') \to (S, B)$ of conical rings gives rise to a morphism of schemes
\begin{align*}
   \bProj(\Phi)\colon \bigcup_{f\in\Rel^{D_S}_{B}(S)} \Spec(S_{(f)}) \to \bigcup_{g\in\Rel^{D_R}_{B'}(R)} \Spec(S_{(g)}).
\end{align*}

A rationally relevant morphism $\Phi \colon (R, R_+) \to (S, B)$, $B \subsetneq S_+$, gives rise to a morphism of schemes
\begin{align*}
    \bProj^\ast(\Phi) \colon \Proj^{D_S}_{B_S}(S) = \bigcup_{f\in\varphi(\Rel^{D_R}(R))} \Spec(S_{(f)}) \to \bigcup_{g\in\Rel^{D_R}(R)} \Spec(S_{(g)}) = \Proj^{D_R}(R).
\end{align*}

A rational map $\Phi\colon (R, B_R) \dashrightarrow (S, B_S)$ of conical rings gives rise to a morphism of schemes
\begin{align*}
    \bProj_{\text{rat}}(\Phi) \colon \bigcup_{f\in\varphi(B_R)} \Spec(S_{(f)}) \to \bigcup_{g \in B_R} \Spec(S_{(g)}) .
\end{align*}
\end{proposition}

\begin{proof}
In each case, the morphism is obtained by gluing the compatible local morphisms induced by $R_{(g)} \to S_{(\varphi(g))}$. 
\end{proof}

With the theory developed so far, we can now give morphisms on (common) open subsets.

\begin{example}\label{ex:standard_ex_maps_between}
This example builds on computations done in our previous papers, which can also be found in the author's thesis \footnote{\url{https://sites.google.com/view/felixgoebler/home/research}}. 
 For $S = \BC[x,y,z,w]$ with degrees $e_1, e_1, e_1+e_2$ and $e_2$, we saw that the subset $B = (xz,yz, xw, yw) \subseteq S_+$ gives the blowup of $\BP^2$ in a point, i.e.\ $\Bl_p(\BP^2) = \Proj^D_B(S)$ (see \cite{paper2}, Example 2.4).
    We want to use Lemma~\ref{lem:subring_localization} to compute a subring $S' \subseteq S$ such that $\Bl_p(\BP^2) = \Proj^D(S')$. We know that $\Bl_p(\BP^2) \subseteq \BP^2\times\BP^1 = \Proj^{\BZ^2}(R)$ for $R = \BC[a, b, c, d, e]$, where $\deg(a) = \deg(b) = (1,0)$ and $\deg(c), \deg(d), \deg(e) = (1, 1)$. Thus, we define the map
\begin{align*}
    \varphi \colon R \to S, \ a \mapsto x, b \mapsto y, c \mapsto z, d \mapsto xw, e \mapsto yw ,
\end{align*}
which is clearly a morphism of multigraded rings, giving us $S' = \varphi(R) = \BC[x, y, z, xw, yw] \subsetneq S$. Then the relevant elements of $S'$ are generated by $(xz, yz, xw, yw)$ and thus $\Proj^D(S') = \Proj^D_B(S) = \Bl_p(\BP^2)$ by Corollary~\ref{cor:subring_loc}. 

In particular, $S'$ can be seen as a refinement of the Cox ring of $\Bl_p(\BP_\BC^2)$ and therefore as the actual coordinate ring. 
Note that we might expect $\Proj^D(S')$ to be given by a $\BN$-graded Proj construction, or at least by a product of weighted projective spaces, as it is separated. By the above construction, we see that
\begin{align*}
    \Bl_p(\BP^2) \equ \Proj^{\BN\times\BN}(\BC[a,b,c,d,e]/(ae-bd)).
\end{align*}
\end{example}

%% file: Pre_Toric_Proj.tex
The classical $\BN$-graded Proj construction is a fundamental tool to gain a better understanding of projective spaces. Even though Brenner--Schröer generalized this construction in 2000, there are very few subsequent papers (notably \cite{KU, KSU, May}) on the multigraded Proj construction.  
We make precise a reverse Cox correspondence: the multigraded polynomial ring $S$ together with an irrelevant subset $B$ encodes the same data as a simplicial system of fans \(\mathcal{S}\).
We will analyze how noetherian multigraded polynomial rings $S$ give rise to simplicial systems of fans $\CS$, and in turn, how simplicial toric prevarieties give rise to multigraded noetherian polynomial rings. 

While the latter is essentially (a refinement of) the Cox construction (cf.\ \cite{Cox}, Section 1), the first one can be seen as some kind of reverse construction. We want to emphasize that Lemma~\ref{lem:subring_localization} allows us to identify every Cox ring $S$ with given irrelevant ideal $B \subseteq S_+$, with a subring $S'$ such that $(S')_+ = B$.
In particular, $S'$ can be seen as a refinement of the Cox ring $S$ (and therefore may be seen as the actual coordinate ring) and the abovementioned assignments are inverse to each other in this case.

Alternatively, we show that we can identify subsets of the full irrelevant ideal $S_+$ with a subsystem of fans of $\CS$. This implies that a simplicial toric prevariety $X$ is fully determined by a triple $(D, S, B)$, where $(D, S)$ is a multigraded polynomial ring and $B\subseteq S_+$ is a ideal in $S$ (see Corollary~\ref{cor:toric_var_via_ring}).

For this whole chapter, $S = \BC[T_1,\ldots, T_n]$ will always denote a multigraded polynomial ring that is graded by a finitely generated abelian group $D$, where $n > r = \rank(D)$, and $X$ will be a normal toric variety, mostly assumed to be simplicial.

\subsection{Maximal Cones}\label{sec:max_cones}
We already know that relevant $f\in S$ give rise to maximal dimensional cones $\CC_D(f)$ in the vector space $D_\BR$.
In particular, we showed that intersections of those weight cones define a chamber structure on $\sigma(S)$ in \cite{paper2}, Corollary 2.12.
In this section, we will define another maximal cone $\sigma_f$ that is associated with relevant $f$ for polynomial rings and lies in a lattice associated with the grading:  
Let $\gamma \colon \BZ^n \to D,  e_i \mapsto \deg(T_i)$ be $\BZ$-linear and $M_S := \ker(\gamma)$. Then we take $N_S$ to be the dual, i.e.\ $N_S := \Hom_\BZ(M_S, \BZ)$. If there is no ambiguity, we might drop the ring in the index.
The following construction is from \cite{BS} (cf.\ Remark 3.7).
We will spend a large part of this chapter elaborating on this construction and drawing conclusions from it.

\begin{definition}
    Let $I = \{1, \ldots, n\}$. For each subset $J \subseteq I$, let $\sigma_J\subseteq N_\BR$ be the convex cone generated by the projections $\pr_i|_M\colon \BZ^n \to \BZ \in \Hom_\BZ(M, \BZ)$ restricted to $M$, $i \in J$.
\end{definition}

Now by \cite{paper1}, Lemma 1.20 (also see \cite{CRB}, Proposition 2.1.3.4), every element $f \in \Gen^D(S)$ has a unique factorization
\begin{align*}
  f \equ T_1^{\epsilon_1} \cdot \ldots \cdot T_n^{\epsilon_n} ,
\end{align*}
where $\epsilon_i \in \{0, 1\}$ and $\epsilon_i \neq 0$ for exactly $r$ indices, or all $\epsilon_i = 0$ if $1$ is relevant.
Thus there is a well-defined support $J_f :=\{i \in I\mid \epsilon_i \neq 0\}$. Note that by \cite{BS}, Proposition 3.4 it follows that the affine piece $\Spec(S_{(f)})$ is given by the simplicial toric variety $\Spec(\BC[M_{J_f}])$, where 
\begin{align*}
    M_{J_f} \deq (\BZ^J \oplus \BN^{I-J}) \cap M \ \subseteq \ \BZ^n .
\end{align*}

In particular, we have
\begin{align}\label{eq:cone_def}
    M_{J_f} \equ \sigma_{I-J_f}^\vee \cap M,
\end{align}
and we denote $\sigma_f := \sigma_{I-J_f}$.

\begin{remark}\label{rem:cone_compute}
    Let $\{m_1, \ldots, m_k\}$ be a basis of $M$ (in other words, the $m_i$ form a Hilbert basis of the toric variety with maximal-dimensional cones $\sigma_f$). A different choice of a basis will just permute the cones. Then
    \begin{align*}
        \sigma_{I-J_f} \equ \Cone\left( (\pr_i(m_1), \ldots, \pr_i(m_k)) \mid i \in I-J_f\right) \ .
    \end{align*}
\end{remark}

\begin{remark}\label{rem:cones_connect}
    The cones $\sigma_f$ and $\CC_D(f)$ for $f \in \Gen^D(S)$ are related as follows. Let $m_1, \ldots, m_s$ be a choice of minimal generators of $M$ and define $\CD$ to be the matrix given by the degrees of $S$, i.e.\ each column of $A$ corresponds to a degree of a variable of $S$. Let $\CB$ be the matrix with columns $m_i$. Then $\CD \CB = 0$. This is a special case of \emph{Gale duality} (cf.\ \cite{CRB}, Remark 2.2.1.4, see also \cite{CLS}, §14.3). 
\end{remark}

\begin{example}\label{ex:max_cones}
We want to compute the maximal cones for some examples we already discussed (here or in our previous papers).
    \begin{enumerate}[label=(\arabic*)]        
        \item Let $S = \BC[x, y, z, w]$ and $D = \BZ^2$, where $\deg(x) = \deg(y) = (1, 0)$, $\deg(z) = (1, 1)$ and $\deg(w) = (0, 1)$. Choosing generators $M = [(1, 0, -1, 1), (0, 1, -1, 1)]$, we conclude that the maximal-dimensional cones are given by
\begin{align*}
    \sigma_{xw} &\equ \Cone(e_2, -e_1-e_2), \\
    \sigma_{yw} &\equ \Cone(e_1, -e_1-e_2), \\
    \sigma_{zw} &\equ \Cone(e_1, e_2), \\
    \sigma_{xz} &\equ \Cone(e_2, e_1+e_2), \\
    \sigma_{yz} &\equ \Cone(e_1, e_1+e_2).
\end{align*}
    Note that $\sigma_{zw} = \sigma_{xz}\cup\sigma_{yz}$ and that all cones have a common face with $\sigma_{wz}$.

    \item Let $S = \BC[x, y, z]$ and $D= \BZ\times \BZ/2\BZ$, where $\deg(x) = (1, 0)$, $\deg(y) = (0, 1)$ and $\deg(z) = (1, 1)$. Compared to (2), we have much more equations in $D$, namely
    \begin{itemize}
        \item $\deg(xy) \equ \deg(z)$ (this is the only equation in (2)),
        \item $\deg(x^2) \equ \deg(z^2)$,
        \item $\deg(y^2) \equ 0$ (this equation follows from the first and second equations),
        \item $\deg(x) \equ \deg(yz)$ (follows from multiplying the first equation with $\deg(y)$ and using $\deg(y^2) = 0$).
    \end{itemize}
    Thus we compute $M_S = [(1, 1, -1), (2, 0, -2)]$ and the cones are
    \begin{align*}
        \sigma_x \equ \Cone\left( \begin{pmatrix}
            1\\0
        \end{pmatrix}, \begin{pmatrix}
            -1\\-2
        \end{pmatrix}\right) \text{ and } \sigma_z \equ \Cone\left( \begin{pmatrix}
            1\\2
        \end{pmatrix}, \begin{pmatrix}
            1\\0
        \end{pmatrix}\right).
    \end{align*}
    \end{enumerate}
\end{example}

\begin{remark}
\begin{enumerate}[label=(\alph*)]
    \item  Note that $M_S = M_R$ in (1) for $R = \BC[x, y, z]$, where $\deg(x) = (2, 1)$, $\deg(y) = (1, 2)$ and $\deg(z) = (3, 3)$. Therefore, we expect that there are infinitely many gradings giving lattices isomorphic to $M, N$. Also note that $(R, R_+)$ and $(S, S_+)$ are isomorphic as conical rings.
    In particular, $\Proj^D(S)$ does not depend on the exact numbers $\deg(T_i)$ but on the linear dependency type (see \cite{paper3}, Theorem 3.10), encoded by $\ker(\gamma)$.

    \item By \cite{paper1}, Example 1.28 we also know that $M = [(1, 2, -2)]$ corresponds to $S'$ and $\Proj^D(S')= \Proj^D(S)$, so $\Proj^D(S)$ does also not depend on the exact generators of $\ker(\gamma)$. Instead, we only need to know which entries are positive, zero, or negative.

    \item Also note that $\varphi$ from Example~\ref{ex:maps_special} induces a map $G\colon M_R \to M_S$ of lattices, mapping generators of $M_R$ to generators of $M_S$. Hence, there is an induced map $F \colon N_R \to N_S$ of lattices. We will look closer at this connection between graded ring morphisms and lattice morphisms in the next section.
\end{enumerate}   
\end{remark}

\begin{example}\label{ex:Z/2Z_fan}
    We want to compute the maximal-dimensional cones in a situation where the grading group is finite. Let the grading be given by the map $\gamma\colon \BZ \to \BZ/2\BZ$, $1 \mapsto \overline{1}$ (cf.\ \cite{paper3}, Example 1.5), so $M_S = 2\BZ$. Since $1$ is relevant and $J_1 = \emptyset$, we have $\sigma_1 = 2 \BN$. 
    Thus $M_{J_1} = 2\BZ \cap 2 \BN = 2\BN$. 
    In particular, $\Proj^{\BZ/2\BZ}(S) = \Spec(M_{J_1}) = \Spec(\BR[x^2]) = \BA^1_{S_0}$ is affine and hence separated.
\end{example}

\begin{lemma}
    Let $f \in \Gen^D(S)$ be relevant. Then $\sigma_f \subseteq (N_S)_\BR$ is of maximal dimension.
\end{lemma}

\begin{proof}
    By Remark~\ref{rem:cones_connect}, $\sigma_f$ is generated by $|I\setminus J_f|$ elements. In particular, since all $f \in \Gen^D(S)$ have the same length, $|I\setminus J_f|$ is constant. Hence, all $\sigma_f$ have the same dimension. Now by construction, $|I\setminus J_f| = n-r = \rank(M_S)$, and thus $\sigma_f$ is maximal.
\end{proof}

\begin{remark}\label{rem:cox_ideal}
Let the maximal cones of $S$ form a simplicial fan $\Delta$ such that $\Delta(1)$ spans $(N_S)_\BR$. Then $\sigma_f$, where $f$ is a generator of $S_+$, coincides with the cone giving the monomial $x^{\widetilde{\sigma_f}}$ in \cite{Cox}.
    In particular, the Cox construction is just the \cite{BS} construction but with a (possibly) different irrelevant ideal. 
\end{remark}

The cones $\CC_D(f)$ and $\sigma(f)$ are connected as follows. 

\begin{lemma}
    Let $f \in S$ be homogeneous. Then $\lint(\CC_D(f) ) \not= \emptyset$ if and only if $\lint(\sigma_f) \not= \emptyset$.
\end{lemma}

\begin{proof}
    Maximality of $\CC_D(f)$ should induce maximality of $\sigma_f$ and vice versa. Let $\lint(\CC_D(f)) \neq \emptyset$, i.e.\ $\CC_D(f)$ is a maximal cone. Thus it has $r$ generators $d_{i_1}, \ldots, d_{i_r} \in \{d_1, \ldots, d_n\}$. In particular, $\sigma_f = \sigma_{I \setminus J_f}$ for $J = \{i_1, \ldots, i_r\}$, that is, $\dim(\sigma_f) = n-r = \rank(M_S)$ is maximal.
    Conversely, let $\lint(\sigma_f) \neq \emptyset$. Then $\sigma_f$ is a maximal cone of dimension $n-r$. In particular, $|J_f| = r$ and $\CC_D(f)$ is maximal.
\end{proof}

In particular, for every collection of relevant $(f_i)_{i\in I}$ such that the union of the $\sigma(f_i)$ forms a fan, the corresponding \emph{GIT cone}, i.e.\ the intersection of all $\CC_D(f_i)$ (cf.\ \cite{paper2}, Definition 2.3 (2)) is maximal-dimensional by the following result:

\begin{proposition}\label{prop:int_CC=sigma}
    Let $f, g \in \Gen^D(S)$, $f \neq g$. Then $\lint(\CC_D(f) \cap \CC_D(g)) \neq \emptyset$ if and only if $\lint(\sigma_f \cap \sigma_g) = \emptyset$.
\end{proposition}

\begin{proof}
\Ra Let $\lint(\CC_D(f) \cap \CC_D(g)) \neq\emptyset$, where $f = \prod_{i\in I} T_i$ and $g = \prod_{j\in J} T_j$ such that $I$ and $J$ are subsets of $\{1, \ldots, n\}$ of cardinality $r$.
Then there exists an element $d' \in \lint(\CC_D(f) \cap \CC_D(g))$ and elements $f', g' \in S$ such that $f\mid f'$ and $g\mid g'$ and $\deg(f') = d' = \deg(g')$.
     Thus there are $\epsilon_i, \delta_j \in \BN_{\ge 1}$, such that
     \begin{align*}
         \sum_{i\in I} \epsilon_i \cdot \deg(T_i) \equ \sum_{j\in J} \delta_j \cdot \deg(T_j) .
     \end{align*}
    In particular, this equation gives rise to an element in $M_S$: For relevant $f = T_1^{\epsilon_1} \cdot \ldots \cdot T_n^{\epsilon_n}$, where $\epsilon_i \in \BN$, consider the vector $v_f = \{\epsilon_1, \ldots \epsilon_n\}$.
    Now by construction, $m' = v_{f'}-v_{g'} \in M_S$, thus we can take it as a generator of $M_S$. In particular, for all $i \in \{1, \ldots, n\}\setminus I$ there is a $j \in\{1, \ldots, n\}\setminus J$ with $\pr_i(m') = \pr_j(m')$, so that the image of $m'$ is either a generator of both $\sigma_f$ and $\sigma_g$ or that $m'$ lies on the boundary of $\sigma_f$ or $\sigma_g$. In particular, $\lint(\sigma_f \cap \sigma_g) = \emptyset$
    \\

    \La Conversely, let $\lint(\sigma_f\cap \sigma_g) = \emptyset$, where without loss of generality, $M_S = \langle m_1, \ldots, m_{n-r} \rangle$ and let $A$ be the matrix with columns $m_i$. 
    
 
    Then either $\sigma_f$ and $\sigma_g$ share exactly one generator or the generators of $\sigma_f$ and $\sigma_f$ are disjoint.
    In the first case, we know that $f$ and $g$ have at least one common homogeneous divisor $x$, giving rise to a (possibly trivial) linear dependency between homogeneous divisors of $f$ and homogeneous divisors of $g$ (cf.\ \cite{paper3}, Definition 3.7). Thus $\lint(\CC_D(f) \cap \lint(\CC_D(g)) \neq \emptyset$.
    In the second case, $J_f$ and $J_g$ are disjoint, or in other words, we may assume $J_g \subseteq I \setminus J_f$. Since $M_S$ is not trivial, there has to be a linear dependency between the homogeneous divisors of $f$ and $g$, and again, $\lint(\CC_D(f) \cap \lint(\CC_D(g)) \neq \emptyset$.
\end{proof}


\subsection{Systems of Fans} \label{sec:sys_fans}
We begin with a summary of the standard theory on systems of fans, drawing from \cite{ANH}.

When we speak of a cone $\sigma$ in $N_S$ we always think of a strictly convex simplicial rational polyhedral cone in $(N_S)_\BR$.
We denote a face $\tau$ of $\sigma$ by $\tau \preceq \sigma$. If $\Delta'$ is a subfan of $\Delta$, we will write $\Delta' \preceq \Delta$. A fan is called \emph{irreducible} if it consists of all the faces of a single cone.

\begin{definition}\label{def:sys_fans}
    Let $I$ be a finite index set. A collection $\CS = \left(\Delta_{ij}\right)_{i, j \in I}$ of fans in $(N_S)_\BR$ is called a \emph{system of fans} if for all $i, j, k \in I$:
    \begin{enumerate}[label=(\roman*)]
        \item $\Delta_{ij} \equ \Delta_{ji}$,
        \item $\Delta_{ij} \cap \Delta_{jk} \ \preceq \ \Delta_{ik}$.
    \end{enumerate}
\end{definition}

There is a natural way to associate a toric prevariety with a given system of fans. For this, just take $X_i$ to be the toric variety corresponding to the fan $\Delta_{ii}$, while (ii) ensures that the gluing $X_\CS= \cup_{i \in I} X_i$ is well behaved.

Now let $\Delta_f$, where $f \in \Gen_B^D(S)$, be the fan of faces of $\sigma_f$ and for $f, g \in \Gen^D_B(S)$ the intersection will always be $\Delta_{fg} = \{\{0\}\}$. Then it holds (cf.\ \cite{BS}, Proposition 3.4):

\begin{theorem}[Multigraded rings yield simplicial toric prevarieties]\label{thm:ring_to_prev}\makebox{}{}\\
    Let $\CS_B = (\Delta_f, \Delta_{fg})_{f,g\in\Gen_B^D}$ be the simplicial system of fans associated to the conical ring $(S, B)$. Then $\Proj^D_B(S) = X_{\CS_B}$ is a simplicial toric prevariety.
\end{theorem}

\begin{proof}
By definition of relevant elements, $\sigma_f$ is automatically simplicial. Then apply \cite{ANH}, Theorem 3.6.
\end{proof}

A system of fans is a fan if the system collapses to a single fan.
We can immediately deduce.

\begin{corollary}\label{cor:sep_crit_toric_version}
    Let $S$ be a noetherian polynomial ring over a field, $B \subseteq S_+$. Then $\Proj^D_B(S)$ is separated if and only if $\CS_B$ is a fan.
\end{corollary}

In particular, applying Prop~\ref{prop:int_CC=sigma} to \cite{paper2}, Proposition 1.1 yields:

\begin{corollary}\label{cor:sep_algebraic_geom}
Let $(S, B)$ be a conical ring and $f, g \in \Gen^D_B(S)$. Then the following statements are equivalent: 
    \begin{enumerate}[label=(\alph*)]
        \item $\mu_{(fg)}$ is surjective.
        \item $\lint\left(\CC_D(f) \cap \CC_D(g)\right) \ \neq \ \emptyset$.
        \item $\lint(\sigma_f\cap\sigma_g) \equ \emptyset$.
    \end{enumerate}
\end{corollary}

\begin{remark}
    Let $X$ be a quasiprojective toric prevariety with full-dimensional convex support. Then Corollary~\ref{cor:sep_algebraic_geom} implies that each GIT cone $\lambda(d)$ corresponds to an open toric subvariety in $X$.
\end{remark}

Next, we show that $\Proj^D(S)$ is a good prequotient in the sense of \cite{ANH}, Definition 6.1: 

\begin{definition}\label{def:toric_prequot}
    Let $k$ be an algebraically closed field, $X$ an algebraic prevariety over $k$ and let a group $\BG$ act on $X$ by means of a regular map $\BG\times X\to X$. 
    A $\BG$-invariant regular map $p\colon X\to Y$ onto a prevariety $Y$ is called a \emph{good prequotient} for the action of $\BG$ on $X$, if
    \begin{enumerate}[label=(\arabic*)]
        \item $p$ is an affine map.
        \item for all open $V \subseteq Y$, $\CO_Y(V) = \CO_X(p^{-1}(V))^G$.
    \end{enumerate}
\end{definition}

Applying \cite{ANH}, Proposition 6.4 and Corollary 6.5., $\Proj^D(S)$ is even a good prequotient in the category of toric prevarieties, i.e.\ a \emph{toric prequotient} (cf.\ \cite{ANH}, Definition 7.1).

\begin{corollary}[GIT Quotient]\label{cor:good_preq}
    $\Proj^D(S)$ is a good prequotient for the action of $G$ on $\Spec(S)$. In fact, it is even a \emph{toric prequotient}.
\end{corollary}

\begin{proof}
    By construction of $\Proj^D(S)$, we already know that 
    \begin{align*}
        \CO_{\Proj^D(S)}(V) \equ \CO_{\Spec(S)}(p^{-1}(V))^G
    \end{align*}
    for $V \subseteq \Proj^D(S)$ open and $p = \pi_+$ (see \cite{paper1}), and $p$ is surjective by \cite{paper1}, Proposition 2.13 (a). A basic affine open subset of $\Proj^D(S)$ is given by $\Spec(S_{(f)})$ for some relevant $f \in S$. Thus $p^{-1}(\Spec(S_{(f))}) = \Spec(S_f)$ is affine. In particular, $\Proj^D(S)$ is a good prequotient for the action of $G$ on $\Spec(S)$. By the universal property of $p$, clearly every $G$-invariant regular map from $\Spec(S)$ to some prescheme factors uniquely through $p$. Hence $\Proj^D(S)$ is even a categorical prequotient (also cf.\ \cite{ANH}, Proposition 6.4 and Corollary 6.5). As $\Proj^D(S)$ is a simplicial toric prevariety by Theorem~\ref{thm:ring_to_prev}, the claim follows.
\end{proof}

There are different charts giving the same toric prevariety, namely the covering by maximal $G$-stable separated charts and the covering by maximal $G$-stable affine charts. The latter corresponds to a system $\CS$ where every fan is irreducible. Such a system will be called \emph{affine}.
It is immediate that the system from Theorem~\ref{thm:ring_to_prev} is affine, i.e.\ the charts given by the collection of $f \in \Gen^D_B(S)$ yield an affine system.

\begin{example}\label{ex:sys_fans}
    We want to compute (some of) the toric prevarieties corresponding to the affine systems from Example~\ref{ex:max_cones}. 
First note that for $B_1 = \{xw, yw, zw\}$ and $B_2 = \{xw, yw, xz, yz\}$ it holds
    \begin{align*}
        X_{\CS, B_1} \equ \BP_\BC^2\ \text{  and  }\ X_{\CS, B_2} \equ \Bl_p(\BP_\BC^2),
    \end{align*}
    where $p = (1:0:0)$. Thus we can view $X_\CS$ as a $\BP_\BC^2$ glued with $\Bl_0(\BA_\BC^2)$ (cf.\ \cite{CLS}, Example 3.1.14) along $\BA_\BC^2\setminus\{0\}$ as well as $\Bl_p(\BP_\BC^2)$ glued with $\BA_\BC^2$ along $\BA_\BC^2\setminus\{0\}$ (i.e.\ this is the covering by maximal $G$-stable affine charts). 
    Note that we can also view $X_\CS$ as $\BP_\BC^2$ glued with $\Bl_p(\BP_\BC^2)$ along $U_1 \cup U_2$, where $p \in U_0$ and $U_i$ are the standard open affines of $\BP_\BC^2$ (this is the covering by maximal $G$-stable separated charts). \\

\end{example}

We want to emphasize that the irrelevant ideals by \cite{Cox} and \cite{BS} differ in general. This is one of the major difficulties in understanding the relation between the multigraded Proj and the Cox ring construction.

\begin{example}
    By \cite{CLS}, Example 5.2.3, the toric variety $\Bl_0(\BA_\BC^2)$ has total coordinate ring $S = \BC[x, y, z]$, where $\deg(x), \deg(y) = 1$ and $\deg(z) = -1$. Choosing $M = \left[\begin{pmatrix}
        0\\1\\1
    \end{pmatrix}, \begin{pmatrix}
        1\\0\\1
    \end{pmatrix}\right]$, we see that $\sigma_x = \Cone(e_1, e_1+e_2)$ and $\sigma_y = \Cone(e_2, e_1+e_2)$ give the toric variety we started with. However, as Brenner--Schröer Proj takes all maximal cones into account, we also have $\sigma_z = \Cone(e_1, e_2)$. In particular, Cox construction coincides with the conical ring $(S, B = (x, y))$, whereas the conical ring belonging to the Brenner--Schröer Proj is given by $(S, S_+ = (x, y, z))$. 
    
    As the latter contains a projective line with double origin (see \cite{paper1}, Example 1.28), $\Proj^D(S)$ is clearly non-separated. Note that $\Bl_0(\BA^2_\BC)$ is also given in Example~\ref{ex:sys_fans} by the open subset $\Spec(S_{(xz)}) \cup \Spec(S_{(yz)}) \subsetneq \Spec(S_{(z)})$, i.e.\ for $R = \BC[x, y, z, w]$ like in the previous example and  $B' = (xz, yz)$ we can identify (using Corollary~\ref{cor:subring_loc})
    \begin{align*}
      \Proj^{\BZ^2}_{B'}(R)  \equ \Bl_0(\BA^2)  \equ \Proj^{\BZ}_B(S).
    \end{align*}
    To see the corresponding isomorphisms of conical rings, we take the morphism $\varphi\colon S \to R_z$, $x \mapsto x$, $y \mapsto y$, and $z \mapsto \frac{w}{z}$ of multigraded rings. Now we just have to take the ring given by the image of $S$ in $R_z$, which is $S' := \BC[x, y, \frac{w}{z}] \subset R_z$, where clearly $\deg(x) = \deg(y) = 1$ and $\deg(\frac{w}{z}) = \deg(w) - \deg(z) = -1$. Thus the inverse of $\varphi$ is given by $\psi\colon S' \to S$, where $x\mapsto x$, $y\mapsto y$ and $\frac{w}{z}\mapsto z$, i.e.\ $S \cong S'$.
\end{example}

\begin{remark}\label{rem:negative_grading}
    In the previous example, we saw that it seems to be possible to rephrase a grading where a variable $T_i$ satisfies $\deg_D(T_i) < 0$ (or where $\deg(T_i) \in \BZ^r$ has a negative entry). In general, let there be $f, g \in S$ such that $\deg(f) = - \deg(g)$. Then $\BC[f, g] \to \BC[f', g', h]$, where $\deg(f') = (\deg(f), 0)$, $\deg(g') = (0, \deg(g))$ and $\deg(h) = (\deg(f), \deg(g))$ gives the isomorphism $\BC[f, g] \cong \BC[f', \frac{g'}{h'}]$. Hence $\deg(\frac{g}{h}) = -\deg(f)$.
\end{remark}

We want to study the $\bProj$ functor in this setting (cf.\ Proposition~\ref{prop:Proj_B^D}).

\begin{corollary}\label{cor:functor_special_setting}
    Let $S$ be a factorially graded noetherian polynomial ring over a field. Then $\bProj$ is a functor that maps $D_S$-graded conical rings $(S, B)$ to simplicial toric prevarieties $\Proj^D_B(S)$, and rational maps of conical rings are mapped to morphisms of toric prevarieties.
\end{corollary}

\begin{proof}
The first statement directly follows from Theorem~\ref{thm:ring_to_prev}.
    Let $\Phi\colon (R, B_R) \dashrightarrow (S, S_+)$ be a rational map of conical rings and let $B_S := \{f \in \Gen^{D_S}(S) \mid \exists g \in B_R: \varphi(g) \mid f\}$. Then the local maps $R_{(g)} \to S_{(f)} = S_{(\varphi(g))}$ give rise to a collection of morphisms $\Spec(S_{(f)}) \to \Spec(R_{(g)})$ for all $f \in B_S$ and hence to a morphism $\Phi^\ast\colon\Proj^{D_S}_{B_S}(S) \to \Proj^{D_R}_{B_R}(R)$ of toric prevarieties.
\end{proof}

Conversely, let $\CS \subseteq N_\BR$ be a simplicial system of fans 
and let $\CA = (\Delta_{ii}, \Delta_{ij})$ denote the corresponding affine system. Let 
\begin{align*}
    \Omega_{(1)} \deq \{(\rho, i) \mid \rho \in \Delta_{ii}(1)\}
\end{align*}
be the set of one-dimensional cones of $\CS$. We want a ray $\rho$ to occur $k$-times if and only if the cone $\Delta_{ii}$ occurs $k$-times. 
The collection of one-dimensional cones (with extra occurrences) will be denoted by $\CS(1)$. Let $n = |\CS(1)|$ be the number of variables of $S$, $r = \rank(D)$ and hence $\dim(N_\BR) = n-r$. 

Now let $\CB$ be the matrix where the rows are given by the ray generators of the one-dimensional cones in $\CA$, i.e. $\CB \in \BZ^{n \times (n-r)}$, viewed as a map
\begin{align*}
    \varphi\colon \BZ^{n} \to \BZ^{n-r},\ x \mapsto x^T \CB \in \BC^{1\times (n-r}) 
\end{align*}
and
\begin{align*}
    S \deq \BC[x_\rho \mid \rho \in \CS(1)].
\end{align*}
We want to construct a matrix $\CD$ such that $\CD \CB= 0$.
Clearly, we just have to take all the vectors lying in the kernel of $\varphi$.
Since $\varphi$ has rank $n-r$, its kernel is free of rank $r$, giving us $\CD \in \BZ^{r \times n}$.

\begin{remark}[Cox Ring]\label{rem:cox_ring}
    In many cases, the ring $S$ coincides with the so-called \emph{Cox} ring of the toric prevariety associated to the simplicial system of fans $\CS$, which is graded by the finitely generated abelian group $\Cl(X)$ (cf.\  \cite{Cox}, §1). However, it need not be the case. Taking $S = \BC[x,y,z, w]$ from Example~\ref{ex:sys_fans}, we see that $S' = \BC[x,y,z,xw,yw]\subseteq S$ satisfies $\Proj^{D_S}(S') = \Bl_p(\BP^2)$. But most importantly, $S'$ is \textbf{not} a Cox ring of $\Bl_p(\BP^2)$ in the sense of \cite{HK}, Definition 2.6.
\end{remark}

Since we have to show the existence of such a matrix $\CD$, we use the Cox ring proof from \cite{CLS} to prove it.

\begin{proposition}\label{prop:cox_dual}
    The above construction coincides with the Cox construction (cf.\ \cite{CLS}, Theorem 4.1.3). In particular, the class group of $X_\CS$ is isomorphic to the kernel of 
    \begin{align*}
        \varphi_\CB\colon \BZ^{r \times n} \to \BZ^{r \times (n-r)}, \ \CD \mapsto \CD \cdot \CB.
    \end{align*}
\end{proposition}

\begin{proof}
By construction, $S$ is graded by $\CD$ and $\Proj^D(S) = X_\CS$ is a toric prevariety. Hence, the class group of $X_\CS$ is isomorphic to the group $D$ corresponding to $\CD$. 
Here we are using arguments dual to those in \cite{CLS}, Theorem 4.1.3, where the class group is described as a cokernel rather than a kernel.
By abuse of notation, we identify $G$ with its $S_0$-points  $G(S_0) = \Hom_\BZ(D, S_0^\times)$, as it is common in toric geometry. The claim follows.
\end{proof}

Thus $S$ is graded by $\CD$, where the degree of $x_{\rho_i}$ is given by the $i$-th column of $\CD$ (therefore the notation).
In other words, given a system of fans, we can compute an associated conical ring by choosing one variable for each one-dimensional cone (see above) and taking the grading such that the linear dependencies between the variables correspond to the generators of $M$. In particular, the subset of the irrelevant ideal belonging to $S$ is given by the collection of the relevant elements corresponding to the maximal cones of $\CS$:
The irrelevant ideal is given by Cox construction applied to the system of fans $\CS$, i.e.\ we take the relevant elements corresponding to the affine system $\CA$ of $\CS$  (i.e.\ the elements corresponding to the maximal cones of $\CS$) and denote it by $B_\CA$ (cf.\ Remark~\ref{rem:cox_ideal}).

We want to emphasize that $\Proj^D(S)$ really only captures the simplicial part of toric (pre)varieties. This already follows from \cite{paper1}, Proposition 2.13 (1) together with \cite{Cox}, Theorem 2.1.

\begin{example}\label{ex:simplicial_needed}
    Let $N = \BZ^3$ and consider the ray generators $v_1 = (1,0,1)$, $v_2 = (0, 1, 1)$, $v_3 = (-1, 0, 1)$ and $v_4 = (0, -1, 1)$ of the cone $\sigma$. 
    Clearly, $U_\sigma$ is not a simplicial toric variety. 
    The Gale dual is given by
    \begin{align*}
       A \deq \begin{pmatrix}
            1 & 0 & -1 & 0 \\ 0 & 1 & 0 & -1 \\ 1 & 1 & 1 & 1
        \end{pmatrix}.
    \end{align*}
    We want to compute generators of $D:=\Cl(U_\sigma)$, so that we can write the ray generators as linear combinations over that base.
    The isomorphism class of $D = \coker(A^T)$ (cf.\ Proposition~\ref{prop:cox_dual}) is determined by the Smith normal form (SNF), which is computed via greatest common divisors of the determinants of all square minors of $A$.  
    The SNF invariant factors of $A^T$ are given by $(1, 1, 2)$, so that $D = \BZ \oplus \BZ/2\BZ$. 
    Hence, we need to find a generator $r$ for the free part and a generator $t$ for the torsion part in $D$.    
    The element $r = (-1, 1, -1, 1)$ satisfies $A r = 0$ and $mr \not\in \im(A^T)$ for all integers $m \neq 0$ (the equation $A^T(a,b,c) = m (-1,1,-1,1)$ yields $m  = 0$).
    Regarding $t$, it holds $t \not \in \im(A^T)$, and $A^T(-1, -1, -1) = (-2, -2, 0, 0) = 2t$, displaying the torsion (i.e.\ $2t = 0 \in \BZ^4/\im(A^T)$). 
    
    Now we can decompose the ray generators as linear combinations over $r$ and $t$, and compute $S = \BC[x_1, \ldots, x_4]$ with grading $\deg(x_1) = \deg(x_3) = (-1, \overline{1})$ and $\deg(x_2) = \deg(x_4) = (1, \overline{0})$. Hence $S_+ = (x_1, \ldots, x_4)$ and $S_0 = \BC[x_1^2 x_2^2, x_1^2 x_4^2, x_3^2 x_2^2, x_3^2 x_4^2, x_1x_2x_3x_4]$. One quickly checks that the cone corresponding to $\Spec(S_{(x_i)})$ is precisely given by omitting the ray $v_i$, so that $\Spec(S_{(x_i)})$ is a simplicial toric variety. 
    Computing the irrelevant ideal in $\Cox(U_\sigma)$ associated with $\sigma$ using the definition in \cite{Cox}, it must be generated by $1$, as $\sigma$ is a single cone containing all rays. Hence, the affine toric variety $U_\sigma$ is realized as the categorical quotient $\Spec(S)/ \Spec(S_0[D])$ by \cite{Cox}, Theorem 2.1 (ii). On the other hand, $\Proj^D(S)$ is given by the geometric quotient $D(S_+) // \Spec(S_0[D])$, where $D(S_+) = \Spec(S) \setminus\{(0)\}$ removes the origin. In particular, $\Proj^D(S) \subseteq U_\sigma$ can be viewed as the maximal simplicial open subset of $X$. 
\end{example}

The next example illustrates how to `read off' the grading from the given system of fans.

\begin{example} 
Consider the toric prevariety $X$ given in \cite{ANH}, Example 2.4, where $X$ is a $\BP^1$ with origin and infinity doubled. In particular, $M_S = [(1,1,-1,-1)]$.
 In this case we can see that $S$ has $4$ variables, say $x, y, z, w$, such that $\deg(xy) = \deg(zw)$ (since $1+1+(-1)+(-1) = 0$). So every choice of $4$ vectors in $\BR^3$ such that no $3$ of them are linear dependent, will give a ring that induces the system we started with. Thus $\deg(x) = (1, 0, 0)$, $\deg(y) = (0, 1, 1)$, $\deg(z) = (0, 0, 1)$ and $\deg(w) = (1, 1, 0)$ is a ring of this type. If we apply the above construction, however, we get
        $\CD = \begin{pmatrix}
            1 & 0 & 1 & 0 \\
            1 & 0 & 0 & 1 \\
            0 & 1 & 1 & 0 \\
            0 & 1 & 0 & 1
        \end{pmatrix}$. As $\CD$ has rank $3$, we can see that $\CD$ is equivalent to $\CD'$, where $\CD'$ is constructed from $\CD$ by forgetting a row of $\CD$. If we choose to drop the second row, we see that $\CD'$ coincides with the above grading.
\end{example}

Note that the choices of $n$ and $r$ in Proposition~\ref{prop:cox_dual} are minimal, but we can always change to $n+1$ and $r+1$ and still find a model. 

\begin{example}\label{ex:sys_fan_P^2}
    Let $N = \BZ^2$ and $\CS = \Delta$ be the fan of $\BP^2$, where the one dimensional cones are given by $\Cone(e_1)$, $\Cone(e_2)$ and $\Cone(-e_1-e_2)$. We will denote the corresponding variables by $x, y$, and $z$. Thus we get $\CB = \begin{pmatrix}
        1 & 0 \\ 0 & 1 \\ -1 & -1
    \end{pmatrix}$. But then, $A$ has to have $3$ rows and $1$ column, hence $\CD = (1, 1, 1)$. In this case, we just get $\BP^2$ with the classical Proj construction (i.e.\ $R = \BC[x, y, z]$ where $\deg(x)=\deg(y)=\deg(z) = 1$). However, if we add the ray $\Cone(e_1 + e_2)$ and denote the corresponding variable with $w$, we get
    \begin{align*}
        \CB \equ \begin{pmatrix}
        1 & 0 \\ 0 & 1 \\ -1 & -1 \\ 1 & 1
    \end{pmatrix} \ \text{ and } \  \CD \equ \begin{pmatrix}
     1&1&1&0    \\ 0&0&1&1
    \end{pmatrix},
    \end{align*}
    that is, a grading by $\BZ^2$. Then $S = \BC[x, y, z, w]$ and $S_+ = \langle xw, yw, zw, xz, yz \rangle$, where clearly $xz$ and $yz$ are the relevant elements belonging to the new cones $\sigma_{xz}$ and $\sigma_{yz}$ (cf.\ Example~\ref{ex:sys_fans}) arising from the extra ray $w$. In particular, another conical ring associated to $\CS$ is $(S, B)$ for $B = \langle xw, yw, zw \rangle$. Let $\varphi\colon R \to S, x \mapsto xw, y \mapsto yw$ and $z \mapsto zw$. Then $\Phi\colon (R, R_+) \to (\varphi(R), B)$ is an isomorphism of conical rings. 

    Next, we want to get $\BP^2$ from a ring graded by $\BZ^3$. For this, we just add another new ray. First we try $v = \Cone(e_1-e_2)$. It holds
        \begin{align*}
        \CB \equ \begin{pmatrix}
        1 & 0 \\ 0 & 1 \\ -1 & -1 \\ 1 & 1 \\ 1 & -1
    \end{pmatrix} \ \text{ and } \  \CD \equ \begin{pmatrix}
       1&1&1&0  & 0  \\ 0&0&1&1 & 0 \\ -1 & 1 & 1 & 1 &1
    \end{pmatrix}.
    \end{align*}
    Thus let $T = \BC[x, y, z, w, v]$ graded by $\CD$. Then $\BP^2$ is given by the relevant elements $xwv, ywv, zwv \in T_+$. In particular, there are isomorphisms of conical rings $(R, R_+) \to (\varphi(R), B) \to (\psi(S), B')$, where $B ' = \langle xwv, ywv, zwv \rangle$ and $\psi$ maps $xw, yw, zw$ to $xwv, ywv, zwv$. 
    
    We can also find a $\BZ^3$ grading where no negative numbers occur:
    Let $v' = \Cone(-e_2)$, i.e.
\begin{align*}
        \CB \equ \begin{pmatrix}
        1 & 0 \\ 0 & 1 \\ -1 & -1 \\ 1 & 1 \\ 0 & -1
    \end{pmatrix} \ \text{ and } \  \CD \equ \begin{pmatrix}
       1&1&1&0  & 0  \\ 0&0&1&1 & 0 \\ 0 & 1 & 0 & 0 & 1
    \end{pmatrix}.
    \end{align*}
    As before, $\BP^2$ is then given by the relevant elements $xwv', ywv'$, and $zwv'$, and thus there are again isomorphisms of conical rings between the subrings giving $\BP^2$.
\end{example}

\subsection{Maps of systems of fans}

For a system of fans $\CS$, we set
\begin{align*}
    \CF(\CS) \deq \{(\sigma, i) \mid i \in I, \sigma \in \Delta_{ii}\}
\end{align*}
and define an equivalence relation, called \emph{gluing relation}, of labelled cones
\begin{align*}
    (\sigma, i) \ \sim \ (\sigma, j) \ \ \iff \ \ \sigma \in \Delta_{ij}.
\end{align*}
We denote the set of equivalence classes by $\Omega = \Omega(\CS)$ and the class of $(\sigma, i) \in \CF(\CS)$ will be denoted by $[\sigma, i]$.

\begin{remark}
    Note that by Remark 2.7 \cite{ANH} the assignment $[\sigma, i] \to G \cdot x_{[\sigma, i]}$ defines a bijection between $\Omega(\CS)$ and the set of $G$-orbits of $X_\CS$, where $X_i \ni x_{(\sigma, i)} \sim x_{(\tau, j)} \in X_j$ if and only if $(\sigma, i) \sim (\tau, j)$.
    Also note that for distinguished points
    \begin{align}\label{eq:orbit_closed_ANH}
        x_{[\sigma, i]} \in \overline{G \cdot x_{[\tau, j]}} \text{  if and only if  } [\tau, j] \preceq [\sigma, i],
    \end{align}
     where $\preceq$ is a partial ordering on $\Omega(\CS)$, defined by 
    \begin{align*}
        [\tau, j]\preceq [\sigma, i] \ \iff \ \tau \preceq \sigma \text{ is a face and } [\tau, i] = [\tau, j].
    \end{align*}
    We denote the open affine $G$-stable subset of $\Proj^D(S)$ containing $G \cdot x_{[\sigma, i]}$ by $X_{[\sigma, i]}$. It is immediate, that $X_{[\sigma, i]}$ coincides with $\Spec(S_{(f)})$, where $f$ is the relevant element corresponding to the maximal cone $(\sigma_, i)$.
\end{remark}

\begin{definition}\label{def:map_sys_fans}
    Let $\CS = (\Delta_{ij})_{i, j \in I}$ and $\CS' = (\Delta'_{ij})_{i, j \in I'}$ be systems of fans in lattices $N$ and $N'$ repectively. A \emph{map of systems of fans} from $\CS$ to $\CS'$ is a pair $(F, \Ff)$, where $F \colon N \to N'$ is a lattice homomorphism and $\Ff \colon \Omega(\CS) \to \Omega(\CS')$ is a map with the following properties:
    \begin{enumerate}[label=(\roman*)]
        \item If $[\tau, j] \preceq [\sigma, i]$, then $\Ff([\tau, j]) \preceq f([\sigma, i])$, i.e.\ $\Ff$ is order preserving.
        \item If $\Ff([\sigma, i]) = [\sigma', i']$, then $F_\BR(\sigma^\circ) \subset (\sigma')^\circ$.
    \end{enumerate}
\end{definition}

\begin{remark}
    If $\CS'$ is a single fan such that $F_\BR$ maps the cones of $\CS$ to the cones of $\Delta'$, there is a unique map $\Ff$ such that $(F, \Ff)$ is a map of systems of fans.    

If conversely $\CS$ is a single fan and $F_\BR$ satisfies the same conditions as before, then there need not exist a map $(F, \Ff)$ of systems of fans (cf.\ \cite{ANH}, Example 3.3).
\end{remark}

Note that the Orbit-Cone-Correspondence (\cite{CLS}, Theorem 3.2.6) remains true for toric prevarieties. Thus, we can use distinguished points to describe maps of systems of fans.
The following Lemma gives the image of a distinguished point under a toric morphism in terms of the map $\Ff$ of systems of fans.

\begin{lemma}\label{lem:map_on_dinsting_point}
    Let $(F, \Ff)$, where $\Ff\colon \Omega(\CS) \to \Omega(\CS')$, be a map of systems of fans with associated toric morphism $f\colon X_\CS \to X_{\CS'}$. Then for every $[\sigma, i]\in \Omega(\CS)$ it holds
    \begin{align*}
        f(x_{[\sigma, i]}) \equ x'_{\Ff[\sigma, i]} .
    \end{align*}
\end{lemma}

\begin{proof}
    \cite{ANH}, Lemma 3.4.
\end{proof}

Summing up the previous results, we can state the following theorem. Recall that conical rings with rational morphisms of conical rings form a category RatConring (see Definition\ref{def:cat_con}). 

\begin{theorem}\label{thm:functor_conical_rings_toric_prev}
Let $S$ be a noetherian factorially graded polynomial ring over an algebraically closed field.
    The assignments 
    \begin{align*}
        (S, B) &\mapsto \CS_B \\
        (\Phi\colon (R, B_R) \dashrightarrow (S, S_+)) &\mapsto (\bProj_{\text{rat}}(\Phi)\colon\Proj^{D_S}_{B_S}(S) \to \Proj^{D_R}_{B_R}(R))\\
        \text{ and } \\
    \CS &\mapsto (\CD, B_\CA), \\
    (\Phi^\ast\colon \Proj^{D_S}_{B_S}(S) \to \Proj^{D_R}_{B_R}(R)) &\mapsto (\Ff,(\phi_f)_{f \in B_S}) 
\end{align*}
 are inverse to each other, where $\Ff\colon \Rel^{D_S}_{B_S}(S) \to \Rel^{D_R}_{B_R}(R)$ and $\phi_f \colon R_{(\Ff(f))} \to S_{(f)}$.
    Thus, the category $\mathrm{ToricPrev_s}$ of simplicial toric prevarieties (with morphisms of toric prevarieties as morphisms) is anti-equivalent to the category $\mathrm{RatConRing}$.
    In particular, multigraded noetherian polynomial rings serve as coordinate rings for toric prevarieties.
\end{theorem}

\begin{proof}
Let $(S, B)$ be a conical ring. Then $\CS_B$ is an affine simplicial system of fans associated to $(S, B)$ by Theorem~\ref{thm:ring_to_prev}. In particular, the maximal cones of $\CS_B$ are parametrized by the elements $f \in \Gen^D_B(S)$. Thus, the matrix $\CB_B$ of ray generators of $\CS_B$ is given by the rays of $\CS_B$, i.e.\ $\CD_B = S$, and $B_\CA$ is given by the relevant elements corresponding to $\CS_B$ by construction. Thus $(S, B) = (\CD_B, B_\CA)$.

Conversely, let $\CS$ be a system of fans and let $(\CD, B_\CA)$ denote the corresponding conical ring. By definition, the relevant elements of $\CD$ lying in $B_\CA$ correspond to the maximal cones of $\CS$. In particular, $\CS_{B_\CA}$ must have the same maximal cones and hence the same ray generators as $\CS$. Thus $\CS = \CS_{B_\CA}$.

We have already shown that a rational map of conical polynomial rings gives rise to a morphism of prevarieties in Corollary~\ref{cor:functor_special_setting}. Conversely, let $\Tilde{f}\colon X \to Y$ be a morphism of toric prevarieties. Then there are conical rings $(R, B_R)$ and $(S, B_S)$ such that $X = \Proj^{D_S}_{B_S}(S)$ and $Y = \Proj^{D_R}_{B_R}(R)$, giving us the affine systems $\CS_{B_R}$ and $\CS_{B_S}$. In particular, $\Tilde{f}$ gives rise to a map of systems of fans $(F, \Ff)$, where $F \colon N_S \to N_R$ is a lattice morphism (and encodes the data of a group homomorphism $\alpha_F\colon D_R \to D_S$ via duality) and $\Ff\colon \Omega(\CS_{B_S}) \to \Omega(\CS_{B_R})$. Using \cite{paper1}, Lemma 1.30 we distinguish the following cases: 
\begin{description}
    \item[Case 1] For all distinguished points $x_{[\sigma, i]}$ of $X_{\CS_{B_S}}$,  it holds 
    \begin{align*}
        f(x_{[\sigma, i]}) \equ x'_{\Ff[\sigma, i]} \equ x_{[\sigma', i']},
    \end{align*}
    where $[\sigma', i']$ is a maximal cone of $\CS_{B_R}$. Thus $\Ff$ maps maximal cones to maximal cones, that is, relevant elements to relevant elements. We denote the map on relevant elements also by $\Ff$, i.e.\ $\Ff\colon \Rel^{D_S}_{B_S}(S) \to \Rel^{D_R}_{B_R}(R)$. By Proposition 1.3.15 \cite{CLS} and Definition~\ref{def:map_sys_fans} (ii), we get a morphism of affine toric varieties $\Spec(S_{(f)}) \to \Spec(R_{(\Ff(f))})$ for each $f \in B_S$. By standard algebraic geometry, this corresponds to a map of rings $\varphi_f\colon R_{(\Ff(f))} \to S_{(f)}$, as desired.

    \item[Case 2] There is a distinguished point $x_{[\sigma, i]}$ of $X_{\CS_{B_S}}$ such that
    \begin{align*}
        f(x_{[\sigma, i]}) \equ x'_{\Ff[\sigma, i]} \equ x_{[\tau', i']},
    \end{align*}
    where $[\tau', i']$ is not a maximal cone of $\CS_{B_R}$. This contradicts condition (ii) of maps of systems of fans, as in this case the interior of $\tau'$ is empty.
\end{description}
It is left to show that the constructions are inverse to each other on the level of morphisms, which is true by construction (cf.\ \cite{ANH}, Theorem 3.6).
\end{proof}

The previous theorem gives a generalization of \emph{bunched rings}.

\begin{remark}[Bunched Rings]
    In their book, \cite{CRB} construct a so called \emph{bunched ring}, which is an object given by the data $(R, \CF, \Phi)$, where $R$ is an almost free $D$-graded affine $k$-algebra such that $R^\ast = k^\ast$, $\CF$ is a system of pairwise non associated $D$-prime generators for $R$ and $\Phi$ is a true $\CF$-bunch. As they show in §2.2 that a $\CF$-bunch is the Gale dual of a fan, we can easily see how this construction is related to conical rings: Every bunched ring is a special case of a conical ring, where $\CF$ corresponds to the variables of $S$ (i.e.\ generators of $S$ over $S_0$) and $\Phi$ is just the system of fans arising from the relevant elements formed over $\CF$. In particular, the projected $\CF$-faces correspond to the orbit cones by \cite{CRB}, Remark 3.2.2.3, which we know are given by the affine pieces $\Spec(S_{(f)})$ for relevant $f$. Thus, conical rings are a generalization of bunched rings. 
\end{remark}

\begin{corollary}
    If $S$ is a normal noetherian factorially graded polynomial ring over a field, $\Proj^D(S)$ is $\BQ$-factorial. 
\end{corollary}

\begin{proof}
    As $\Proj^D(S)$ is simplicial and normal, we may apply \cite{CLS}, Proposition 4.2.7 (also cf.\ \cite{CRB}, Corollary 3.3.1.9).
\end{proof}

In addition, Theorem~\ref{cor:functor_special_setting} gives an algebraic description of toric (pre-)varieties.

\begin{corollary}\label{cor:toric_var_via_ring}
    Let $X = X_\CS$ be a toric prevariety. Then there exists a subset $B \subseteq \Cox(X)_+$, such that the triple $(D = \Cl(X), S = \Cox(X), B)$ determines the toric variety $X$ uniquely.
    While $\Cox(X)$ is often called the `coordinate ring' of $X$, one should keep in mind that the grading by $D$ together with the 
    irrelevant subset $B$ is an indispensable piece of data. 
    In this sense, one may regard the triple $(D, S, B)$ as the `coordinate ring' of $X$.
\end{corollary}

\begin{proof}
If $X$ is toric, let $B \subseteq S_+$ be the irrelevant ideal constructed in \cite{Cox}. In general, we compute $S_+$ and $\sigma_f$ for all $f \in \Gen^D(S)$. Then we take $B \subseteq S_+$ to contain exactly those $f$ whose $\sigma_f$ is contained in the system of fans $\CS$. 
    Then the toric (pre-)variety $X$ is uniquely determined by
    \begin{align*}
       X \equ  \bigcup_{f \in B} \Spec(S_{(f)}) \subseteq \Proj^D(S),
    \end{align*}
    since each $\Spec(S_{(f)})$ is a simplicial toric prevariety (cf.\ Theorem~\ref{thm:ring_to_prev}).
\end{proof}

The above discussion shows that, for conical rings, Proj construction and Cox ring construction are mutually inverse up to canonical isomorphisms (together with $\Cl(X)$ and the irrelevant ideal $B \subseteq S_+$ defining the given (system of) fans). In particular, we can state:

\begin{corollary}\label{cor:Proj_Cox_inverse}
    Let $X$ be a toric prevariety, $S= \Cox(X)$, $D = \Cl(X)$ and let $B \subseteq S_+$ be the irrelevant ideal as defined by Cox (\cite{Cox}, §1). Then $\Proj^D_B(S) = X$.
    Conversely, let $S$ be a multigraded polynomial ring with finitely many variables, that is graded by $D$. Also, let $B \subseteq S_+$ be a subset and $X = \Proj^D_B(S)$. Then $S = \Cox(X)$, $D= \Cl(X),$ and $B$ is exactly the irrelevant ideal as defined by Cox.
\end{corollary}

\subsection{Subtorus Actions}
In this short section, we want to discuss subtorus actions on toric prevarieties and relate them to the so called \emph{Chow quotient}. The Chow quotient of a toric (pre)variety is characterized by being the smallest toric (pre)variety that maps onto all GIT-quotients.
In fact, we do not have to do much, as it is rather a consequence of known results.
By \cite{ANH}, Theorem 7.2, we already know that every subtorus action on a toric prevariety admits a toric prequotient. Even more, \cite{ANH}, Corollary 8.3 states that every toric prevariety is given by the image of a categorical prequotient of an open toric subvariety of some $k^s$.
However, due to Corollary~\ref{cor:good_preq}, $\Proj^D(S)$ is a good prequotient, so we actually hope that the conditions of \cite{ANH}, Theorem 6.7 are always met, allowing us to identify $X // H =: \Proj^D(S, H)$ with the quotient system of fans constructed in \cite{ANH}, §6: 

Let  $X = \Proj^D(S)$ be a simplicial toric prevariety, given by a simplicial system of fans $\CS$ in the lattice $N_S$ and a subtorus $H \le \Spec(S_0[D])$ corresponds to a primitive sublattice $L \subseteq N_S$ (cf.\ \cite{CLS}, Exercise 1.1.9).
For $f \in \Gen^D(S)$, let $\sigma_{f, L}$ be the largest face of $\sigma_f$ such that $L_\BR \cap \sigma_{f, L}^\circ \neq \emptyset$. Next, let $L' = N_S \cap (L_\BR + \sigma_{f, L})$ and consider the projection $P \colon N_S \to N_S / L'$. Then $\sigma_f' = P_\BR(\sigma_f)$ is a strictly convex simplicial cone in $N/L'$. In particular, the morphism $X_{\sigma_f} \to X_{\sigma_f'}$ associated to $P$ is the algebraic quotient for the action of $H$ on $X_{\sigma_f}$ by \cite{ANH}, Remark 6.6.

The following results strongly depend on the assumption that the toric prevariety $X$ is given in terms of a \emph{full} Proj, i.e.\ $X=\Proj^D(S)$. In particular, we already know that $S$ is not the Cox ring of $X$ in general.

\begin{theorem}
Let $X = \Proj^D(S)$ be a simplicial toric prevariety and $H \le \Spec(S_0[D])$ a subtorus. Then the action of $H$ on $X$ admits a good prequotient.
 In particular,
 \begin{align*}
    \Proj^D(S, H) \equ \bigcup_{f \in \Rel^D(S)} \Spec( (S_f)^H ) \equ \bigcup_{f \in \Rel^D(S)} \Spec(S_f) // H,
\end{align*}
corresponds to the system of fans in $N/L'$ given by the collection of $\sigma'_f$ for $f \in \Gen^D(S)$, where $\sigma'_{fg} = \{\{0\}\}$.
\end{theorem}

\begin{proof}
    As by construction of $\CS$ in Section~\ref{sec:sys_fans}, $\sigma_{fg} = \{\{0\}\}$ for all $f, g \in \Gen^D(S)$. Thus, the condition from \cite{ANH}, Theorem 6.7 holds true for trivial reasons. In particular, the good prequotient of $X$ by $H$ is given by the collection of $\sigma'_f$ for $f \in \Gen^D(S)$, where $\sigma'_{fg} = \{\{0\}\}$. By \cite{paper1}, Lemma 1.15 (4) and \cite{CLS}, Proposition 11.2.14, we deduce that the toric variety corresponding to $\sigma'_f$ is $\Spec((S_f)^H)$.
\end{proof}

We get the following generalization of \cite{KSZ}, Theorem 2.1, for free.

\begin{corollary}
 Let $X$ be a simplicial toric prevariety and $S$ be a $D$-graded polynomial ring such that $X = \Proj^D(S)$. Consider a subtorus $H \le \Spec(S_0[D])$.
    
    Then $\Proj^D(S, H)$ coincides with the chow quotient of $X$ by $H$.
    In particular, the Chow quotient of a simplicial toric prevariety is again a simplicial toric prevariety.
\end{corollary}

\begin{proof}
    First note that $S$ is not the Cox ring of $X$, and that $S$ intrinsically contains the information of the toric prevariety $X$, given in terms of the relevant spectrum of $S$.
    The proof essentially reduces to showing that the quotient fan construction in \cite{KSZ}, §1, is a special case of the quotient system of fans construction in \cite{ANH}, Theorem 6.7. But by construction, $\sigma'_f$ for $f \in \Rel^D(S)$ is exactly the closure of the equivalence class of admissible vectors $n \in \sigma$.  
    By the universal property of the Chow quotient and the explicit description of the $\Proj^D(S, H)$ quotient, the two constructions coincide.
\end{proof}